\documentclass[11pt, reqno]{amsart}
\usepackage{amsaddr}
\usepackage{latexsym,amsfonts}
\usepackage{amssymb,amsthm,upref,amscd}
\usepackage[T1]{fontenc}
\usepackage{times}
\usepackage{amscd,bigints}
\usepackage{enumerate}
\usepackage{mathrsfs}
\usepackage{amsmath}
\usepackage{color}
\usepackage{soul}
\usepackage{indentfirst}
\usepackage{comment}
\PassOptionsToPackage{reqno}{amsmath}

\usepackage{mathtools}
\mathtoolsset{showonlyrefs}

\newtheorem{thm}{Theorem}[section]

\newtheorem{defi}{Definition}[section]
\newtheorem{lem}{Lemma}[section]

\newtheorem{rem}{Remark}[section]

\newcommand{\R}{\mathbb{R}}

\numberwithin{equation}{section}
\newcommand{\N}{\mathbb{N}}
\newcommand{\Z}{\mathbb{Z}}

\newcommand{\G}{\mathbb{G}}

\newcommand{\eps}{\epsilon}

\newcommand{\X}{X^{s, p}_0(\Omega)}
\newcommand{\wto}{\rightharpoonup}

\makeatletter \@addtoreset{equation}{section} \makeatother

\usepackage[
left=30mm,
right=30mm,
top=25.4mm,
bottom=25.4mm,
headheight=2.17cm,
headsep=4mm,
footskip=12mm,
heightrounded,]{geometry}

\usepackage[colorlinks=true,urlcolor=blue,
citecolor=black,linkcolor=black,linktocpage,pdfpagelabels,
bookmarksnumbered,bookmarksopen]{hyperref}
\usepackage{hyperref}
\newcounter{const}
\title[Embedding Theorems and Applications]{\small Fractional Morrey-Sobolev type embeddings and nonlocal subelliptic problems with oscillating nonlinearities on stratified Lie groups}

\author{\small Sekhar Ghosh${}^1$, Tianxiang Gou${}^2$, Vishvesh Kumar${}^3$, Vicen\c tiu D. R\u adulescu${}^{4,5,6,7}$}

\thanks{Emails: sekharghosh1234@gmail.com (S. Ghosh), tianxiang.gou@xjtu.edu.cn (T. Gou), vishvesh.kumar@ugent.be (V. Kumar), radulescu@inf.ucv.ro (V.D. R\u adulescu). }
\thanks{ ${}^1$ Department of Mathematics, National Institute of Technology Calicut, Kozhikode 673601, Kerala, India.}
\thanks{ ${}^2$ School of Mathematics and Statistics, Xi’an Jiaotong University, 710049, Xi’an, Shaanxi, China.}
\thanks{ ${}^3$ Department of Mathematics: Analysis, Logic and Discrete Mathematics, Ghent University, Ghent, Belgium.}
\thanks{ ${}^4$ Faculty of Applied Mathematics, AGH University of Krak\'ow, 30-059 Krak\'ow, Poland.}
\thanks{ ${}^5$ Brno University of Technology, Faculty of Electrical Engineering and Communication, Brno 61600, Czech Republic.}
\thanks{ ${}^6$ Department of Mathematics, University of Craiova, 200585 Craiova, Romania “Simion Stoilow" Institute of Mathematics of the Romanian Academy, P.O. Box 1-764, 014700 Bucharest, Romania.}
\thanks{ ${}^7$ Department of Mathematics, Zhejiang Normal University, 321004 Jinhua, Zhejiang, China.}

\begin{document}

\begin{abstract} 
 
In this paper, we establish the fractional Morrey-Sobolev type embeddings on stratified Lie groups. 
This extends and complements the Sobolev type embeddings derived in \cite{GKR}. 
As an application of the results, we study the following nonlocal subelliptic problem,   
\begin{equation} 
\begin{cases} 
(-\Delta_{\mathbb{G}, p})^s u= \lambda \beta(x) g(u) & \text{in} \quad \Omega, \\
u=0\quad & \text{in}\quad \mathbb{G}\backslash \Omega,
\end{cases}
\end{equation}
where $0<s<1<p<\infty$ with $ps\geq Q$, $Q$ is the homogeneous dimension of the stratified Lie group $\mathbb{G}$, $(-\Delta_{\mathbb{G}, p})^s$ is the fractional $p$-sub-Laplacian on $\mathbb{G},$  $\Omega$ is an open bounded subset of $\mathbb{G}$, $ \lambda$ is a positive real parameter, $\beta \in L^\infty(\Omega, \R_{>0})$ and $g \in C(\mathbb{R}, \R) $ oscillates near the origin or at infinity. By using the variational principle of Ricceri, we prove the existence and asymptotic behaviors of infinitely many solutions to the problem under consideration. We emphasize that the results obtained here are also novel for $\mathbb{G}$ being the Heisenberg group and $p=2$.


\end{abstract}

\maketitle
\tableofcontents
\thispagestyle{empty}
\allowdisplaybreaks
\section{Introduction and main results}

During recent years, nonlocal elliptic partial differential equations (PDEs) and related subjects have been well explored in the Euclidean setting. Nonlocal elliptic PDEs and the corresponding Sobolev spaces are experiencing impressive applications in different fields, which involve fractional models like the $\alpha$-stable L\'evy processes in probability, quantum mechanics, finance, optimization and image processing, we refer the readers to \cite{AMRT, DL, Fd, T} and references therein. 

In their seminal work \cite{RS}, Rothschild and Stein demonstrated that a general H\"ormander’s sums of squares of vector fields on manifolds can be approximated by a sub-Laplacian on some stratified Lie group, which accordingly highlights the importance of nilpotent Lie groups in deriving sharp subelliptic estimates for differential operators on manifolds, see also \cite{DGN, Fo1, G, Ro}.   It turns out later that Rothschild and Stein's results are crucial to study PDEs on stratified Lie groups and lead to numerous interesting and promising works which merge Lie group theories with the analysis of PDEs, see for example \cite{BFG, BF, CDG, CCR, DGN, FF,GKR, GKR23, GKR24, MM, RS, RTY} and references therein. On the other hand, stratified Lie groups model physical systems with constrained dynamics, where movement is confined to particular directions within the tangent space (sub-Riemannian geometry, in contrast to Riemannian geometry), as highlighted in Cartan's groundbreaking work \cite{Eli}. Following the aforementioned results, there exists a rapid growth of interest in the study of sub-Laplacians on stratified Lie groups in recent years due to the relevance not only in theoretical settings but also in practical applications, such as mathematical models of crystal materials and human vision, see \cite{Ch, CMS} and references therein.
The study of these Lie groups poses significant challenges, and many fundamental questions regarding their analytical and geometric properties remain unresolved.

The embeddings of Sobolev spaces constitute an important part of functional analysis and geometry, which have a wide range of remarkable applications in the theories associated to PDEs, calculus of variations and mathematical physics. In particular, they
are one of the most important tools to consider solutions to PDEs and play an essential role in the discussion of the existence and regularity of solutions. 
We refer the readers to \cite{DPV} and references therein for a complete survey of the fractional Sobolev spaces $W^{s,p}(\Omega)$ and the properties of the fractional $p$-Laplacian along with applications to nonlocal PDEs in the Euclidean setting, where $0<s<1 \leq p<+\infty$ and $\Omega \subset \R^N$ is an open subset. See also the monographs \cite{Br, Ev} for the relevant topics developed for the classical Sobolev spaces $W^{k,p}(\Omega)$, where $k \in \Z_{>0}$.

The Sobolev spaces (also known as Folland-Stein spaces) on stratified Lie groups were first introduced by Folland \cite{Fo}. Further properties concerning the Sobolev spaces were later developed in the book by Folland and Stein \cite{FS}. 
Recently, Kassymov and Suragan \cite{KS} studied the fractional Sobolev and Hardy type inequalities on homogeneous Lie groups (see also \cite{RS1}). Moreover,  Adimurthi and Mallick \cite{AM} investigated the fractional Sobolev type inequalities and the Sobolev embeddings for weighted fractional Sobolev spaces on the Heisenberg groups, where the compact embeddings were derived for extension domains. The definition of extension domains is given by Definition \ref{defextension}. The extension property of bounded domains plays a crucial role in the consideration of the Sobolev embeddings into Lebesgue spaces, which has been clarified in \cite{DPV}. It is worth mentioning the work due to Zhou \cite{Zh}, where the author studied the characterizations of $(s, p)$-extension domains and embedding domains for fractional Sobolev spaces in Euclidean space $\R^N$. However, to the best of our knowledge, it is unclear whether such characterizations for an arbitrary bounded domain in the framework of stratified Lie groups remain valid. Indeed, since the existence of characteristic points, then the problem of finding classes of extension domains in stratified Lie groups is highly nontrivial and there are essentially no examples for Step 3 and higher, see \cite{CG}. For this reason, one is forced to investigate the embeddings of the fractional Sobolev space $X_0^{s,p}(\Omega)$ on stratified Lie group with vanishing trace, where $X_0^{s,p}(\Omega)$ denotes the fractional Sobolev space defined by \eqref{defx} equipped with the norm given by \eqref{norm}. Lately, when $ps<Q$, Ghosh et al. \cite{GKR} successfully established the Sobolev embedding of $X_0^p(\Omega)$ into Lebesgue spaces, see Theorem \ref{thm0}. While $sp \geq Q$, the problem was left open so far. The first aim of the paper is to prove the embeddings for the case $sp \geq Q$, which extends and complements the embeddings obtained in \cite{GKR}. 

\begin{thm} \cite[Theorem 1]{GKR} \label{thm0}
Let $\G$ be a stratified Lie group of homogeneous dimension $Q$ and $\Omega \subset \G$ be an open subset. Let $0<s<1 \leq p<+\infty$ and $ps<Q$. Then $X_0^{s,p}(\Omega)$ is continuously embedded into $L^r(\Omega)$ for any $p \leq r \leq p_s^*:=\frac{Qp}{Q-sp}$, i.e. there exists a constant $C=C(Q, s, p, \Omega)>0$ such that, for any $u \in X_0^{s, p}(\Omega)$,
\begin{align*}
\|u\|_{L^r(\Omega)} \leq C \|u\|_{X_{0}^{s, p}(\Omega)}.
\end{align*}
Moreover, if $\Omega$ is a bounded domain, then $X_0^{s,p}(\Omega)$ is compactly embedded into $L^r(\Omega)$ for any $1 \leq r < p_s^*$.
\end{thm}

The main results of the paper regarding the embeddings for the case $sp \geq Q$ are addressed as follows, which can be seen as the counterparts of the ones in the Euclidean setting. More precisely, when $Q=ps$, we have the following result.

\begin{thm} \label{thm1}
Let $\G$ be a stratified Lie group of homogeneous dimension $Q$ and $\Omega \subset \G$ be an open subset. Let $0<s<1 \leq p<+\infty$ and $Q=ps$. Then $X_0^{s,p}(\Omega)$ is continuously embedded into $L^r(\Omega)$ for any $p \leq r < + \infty$, i.e. there exists a constant $C=C(Q, s, p, \Omega)>0$ such that, for any $u \in X_0^{s, p}(\Omega)$,
\begin{align} \label{emthm1}
\|u\|_{L^r(\Omega)} \leq C \|u\|_{X_{0}^{s, p}(\Omega)}.
\end{align}
Moreover, if $\Omega$ is a bounded domain, then $X_0^{s,p}(\Omega)$ is compactly embedded into $L^r(\Omega)$ for any $1 \leq r < + \infty$.
\end{thm}


While $Q<ps$, we obtain the following result.

\begin{thm} \label{thm2}
Let $\G$ be a stratified Lie group of homogeneous dimension $Q$ and $\Omega \subset \G$ be an open subset. Let $0<s<1 \leq p<+\infty$ and $Q<ps$. Then $X_0^{s,p}(\Omega)$ is continuously embedded into $C^{0, \alpha}(\Omega)$ for $\alpha=\frac{sp-Q}{p}$, i.e. there exists a constant $C=C(Q, s, p, \Omega)>0$ such that, for any $u \in X_0^{s, p}(\Omega)$,
\begin{align} \label{emthm2}
\|u\|_{C^{0, \alpha}(\Omega)} \leq C \|u\|_{X_{0}^{s, p}(\Omega)}.
\end{align}
Moreover, if $\Omega$ is bounded, then $X_0^{s,p}(\Omega)$ is compactly embedded into $C^{0, \beta}(\overline{\Omega})$ for any $0<\beta<\alpha$.
\end{thm}


To establish Theorems \ref{thm1} and \ref{thm2}, the key lies in making use of group structures such as the group translation and regularization process via group convolution and dilations for this particular setting of stratified Lie groups. 

In the sequel, we will present applications of the embedding theorems. More precisely, as an application of Theorem \ref{thm2}, we are going to establish the existence and asymptotic behaviors of solutions to the following nonlocal subelliptic nonlinear problem,
\begin{equation} \label{pro1intro}
\begin{cases} 
(-\Delta_{\mathbb{G}, p})^s u= \lambda \beta(x) g(u) & \text{in} \quad \Omega, \\
u=0\quad & \text{in}\quad \mathbb{G}\backslash \Omega,
\end{cases}
\end{equation}
where $0<s<1$, $p<1$, $sp>Q$, $Q$ is the homogeneous dimension of the stratified Lie group $\G$, $(-\Delta_{\mathbb{G}, p})^s$ is the fractional $p$-sub-Laplacian on $\mathbb{G}$ defined by
$$
\left(-\Delta_{p, \mathbb{G}}\right)^s u(x):= C(Q, s, p) P. V. \int_{\mathbb{G}}\frac{|u(x)-u(y)|^{p-2}(u(x)-u(y))}{|y^{-1}x|^{Q+ps}} \,dy,
$$
$\Omega$ is an open bounded subset of $\mathbb{G}$, $ \lambda$ is a positive real parameter, $\beta \in L^\infty(\Omega)$ with $\beta_0:= \text{essinf}_{x \in \Omega} \beta (x)>0$ and $g: \mathbb{R} \rightarrow \mathbb{R}$ is a continuous function which oscillates near the origin or at infinity. Here the symbol P.V. stands for the Cauchy principal value.

Nonlocal problems like \eqref{pro1intro} have been intensively considered in the Euclidean setting during recent decades. For the pertain consideration, we refer the readers to \cite{ADM18, BB1, MP, BRS, GMR, KM, KMT, OO, OZ} and references therein. However, as per our knowledge, there exists no literature devoted to the study of solutions to \eqref{pro1intro} in the setting of stratified Lie groups. In the current setting, the main challenge in the discussion is that we need to deal with the nonlocal and nonlinear nature of the operator $\left(-\Delta_{p, \mathbb{G}}\right)^s$, as well as the non-Euclidean geometry feature of the stratified Lie group $\G$. It is worth highlighting that our forthcoming results are also new in the setting of the Heisenberg group even for $p=2.$ We note that for $p=2$, the operator $\left(-\Delta_{p, \mathbb{G}}\right)^s$ corresponds to the fractional sub-Laplacian 
$$
(-\Delta_\G)^s:=\lim_{\epsilon \rightarrow 0} \int_{\G \backslash B(x, \epsilon)} \frac{|u(x)-u(y)|}{|y^{-1}x|^{Q+2s}} \,dy = C(Q, s) P.V. \int_{\G} \frac{|u(x)-u(y)|}{|y^{-1}x|^{Q+2s}}\,dy,
$$
where $C(Q, s)>0$ is a constant. Problems driven by the operator $\left(-\Delta_{\mathbb{G}}\right)^s$, which are connected with various areas of mathematics, have been well studied recently in several intriguing papers \cite{BF13, FMPPS18, GLV, K20, RT16, RT20} along with references cited therein. It was proved in \cite{RT16,RT20} that  the operator $(-\Delta_\G)^s$ is a multiple of a ``conformal invariant" operator on $H$-type groups, which holds significance in CR geometry (see \cite{FMMT15}). In contrast to the Euclidean space,  the operator $(-\Delta_\G)^s$ does not coincide with the standard fractional power $-\Delta_{\G}^s$ of the sublaplacian $-\Delta_{\G}$ in the Heisenberg groups or general $H$-type groups defined by
$$
(- \Delta_{\G}^su)(x):=-\frac{s}{\Gamma (1-s)}\int_0^\infty \frac{1}{t^{1+s}}(H_tu(x)-u(x))\,dt,
$$
where $H_t:=e^{-t \Delta_\G}$ is the heat semigroup constructed by Folland \cite{Fo}.  In the recent paper of Garofalo and Tralli \cite{GT21}, they computed explicitly the fundamental solutions of the nonlocal operators $(-\Delta_\G)^s$ and $(- \Delta_{\G}^su)$ and established some intertwining formulas on $H$-type groups. For the study of the regularity theory of subelliptic nonlocal equations involving $(-\Delta_{\mathbb{G}, p})^s,$ one may consult recent articles \cite{MMPP23, PP22, Pic22, FZ24} and references therein.  
 
To address our results regarding the existence and asymptotic behaviors of solutions to \eqref{pro1intro}, we will introduce some notations. Define
$$
G(\tau):=\int_0^{\tau} g(\xi) \,d\xi, \quad \forall \,\, \tau \in \R,
$$
$$
A_M:= \liminf_{\tau \rightarrow M} \frac{\max_{|\xi| \leq \tau} G(\xi)}{\tau^p}, \quad
B_M:= \limsup_{\tau \rightarrow M} \frac{G(\tau)}{\tau^p}, \quad M \in \left\{0^+, +\infty \right\}.
$$ 
In addition, we define
$$
D
:= \left( \frac{r^{sp}}{2^Q C_{p, Q, s} L^p |\Omega| \omega_Q} \right) \frac{\beta_0}{\|\beta\|_{L^\infty(\Omega)}},
$$
$$ 
\lambda_{1,M}:=\frac{2^Q C_{p, Q, s} \omega_Q}{p r^{sp} \beta_0 B_M}, \quad \lambda_{2,M}:= \frac{1}{p \|\beta\|_{L^\infty(\Omega)} |\Omega| L^p A_M},
$$
where $|\Omega|$ and $\omega_Q$ denote respectively the Haar measure of $\Omega$ and $Q-1$ dimensional surface measure of the unit quasi-sphere with respect quasi-norm $|\cdot|$,
$$
C_{p, Q, s}:= \left(\frac{2^{p(2-s)}}{(1-s)p}+\frac{2^{2+ps}}{ps}\right)\left(1-\frac{1}{2^Q}\right)+\frac{2^{1+ps-Q}}{ps},
$$
$$
r:=\sup_{x \in \Omega} \text{dist}(x, \partial \Omega), \quad 
L:=\sup\left\{\frac{\|u\|_{L^\infty(\Omega)}}{\|u\|_{X_0^{s,p}(\Omega)}}: u \in X_0^{s, p}(\Omega) \backslash \{0\} \right\}.
$$


\begin{thm} \label{mainthmintrnon} 
Let $\G$ be a stratified Lie group of the homogeneous dimension $Q$ and $\Omega \subset \G$ be a bounded open subset. Let $0<s<1$, $p>1$ and $sp>Q$. 
Suppose that 
$
\inf_{\tau \geq 0} G(\tau)=0
$
and
$
A_M<DB_M.
$
Then, for every $\lambda \in (\lambda_{1,M}, \lambda_{2,M})$, 
\eqref{pro1intro} has a sequence of solutions $\{u_{\lambda, n}\} \subset X_0^{s,p}(\Omega)$.  Moreover, there holds that
$$
\lim_{n \rightarrow \infty} \|u_{\lambda, n}\|_{X^{s,p}_0(\Omega)}=+\infty \quad \text{if} \,\, M=+\infty,
$$
$$
\lim_{n \rightarrow \infty} \|u_{\lambda, n}\|_{X^{s,p}_0(\Omega)}= \lim_{n \rightarrow \infty} \|u_{\lambda, n}\|_{L^\infty(\Omega)}=0 \quad \text{if} \,\, M=0^+.
$$
\end{thm}

\begin{rem}
The condition $A_M<DB_M$ implies that the nonlinearity oscillates near the origin or at infinity. It guarantees that $\lambda_{1,M}<\lambda_{2, M}$.
\end{rem}

In our scenario, we cannot adapt the classical Lusternik-Schnirelmann theory to detect the existence of infinitely many solutions to \eqref{pro1intro}. As a result, to establish Theorem \ref{mainthmintrnon}, we will employ Theorem \ref{ricthm} proposed by Bonanno and Bisci \cite{BB, BB1}, which is a refinement of the well-known variational principle due to Ricceri \cite{Ric00}. Here the solutions are indeed local minima of the underlying energy functional in $\X$, which are not of mountain pass type. And we prove the multiplicity of solutions to \eqref{pro1intro} without imposing symmetry and oddness hypotheses on the nonlinearity.


\begin{thm} \label{color1.1}
Let $\G$ be a stratified Lie group of the homogeneous dimension $Q$ and $\Omega \subset \G$ be a bounded open subset. Let $0<s<1$, $p>1$ and $sp>Q$. 
Suppose that $g$ is nonnegative with $g(0)=0$ and
\begin{align} \label{condition}
\liminf_{\tau \rightarrow M} \frac{G(\tau)}{\tau^p}=0, \quad \limsup_{\tau \rightarrow M} \frac{G(\tau)}{\tau^p}=+\infty, \quad M \in \left\{0^+, +\infty \right\}.
\end{align}
Then, for every $\lambda>0$, \eqref{pro1intro} has a sequence of nonnegative solutions $\{u_{\lambda, n}\} \subset X_0^{s,p}(\Omega)$.  Moreover, there holds that
$$
\lim_{n \rightarrow \infty} \|u_{\lambda, n}\|_{X^{s,p}_0(\Omega)}=+\infty \quad \text{if} \,\, M=+\infty,
$$
$$
\lim_{n \rightarrow \infty} \|u_{\lambda, n}\|_{X^{s,p}_0(\Omega)}= \lim_{n \rightarrow \infty} \|u_{\lambda, n}\|_{L^\infty(\Omega)}=0 \quad \text{if} \,\, M=0^+.
$$
%
%
\end{thm}

It follows from the condition \eqref{condition} that $A_M=0$ and $B_M=+\infty$. Then the existence and the asymptotic behaviors of the solutions are beneficial from Theorem \ref{mainthmintrnon}. To check that the solutions are nonnegative, we need to introduce a modified problem by truncating the nonlinearity. 

The paper is organized as follows. In Section \ref{pre}, we present some preliminary results regarding basic notations and concepts connected with stratified Lie groups and also introduce fractional Sobolev spaces in the framework of stratified Lie groups. In Section \ref{embedding}, we examine the embedding theorems and show the proofs of Theorem \ref{thm1} and \ref{thm2}. Finally, Section \ref{application} is devoted to applications of the embedding theorems and contains the proofs of Theorems \ref{mainthmintrnon} and \ref{color1.1}.


%
 %

\section{Preliminaries: Stratified Lie groups and Fractional Sobolev spaces} \label{pre}

This section is devoted to recapitulating some basic notations and concepts related to stratified Lie groups and introducing a class of fractional Sobolev spaces in the framework of stratified Lie groups, which play important roles in the proof of our main findings. We refer to various books, monographs, and papers such as \cite{FS, BLU07, FR, RS1, GKR, GKR24} for the material presented in this section. 

\subsection{Stratified Lie groups}  There are multiple approaches to introducing the concept of stratified Lie groups, we refer to \cite{FS, BLU07} and references therein. We begin with the definition of homogeneous Lie groups due to Folland and Stein \cite{FS}.

\begin{defi}
A Lie group $\G$ $(\mbox{on} \,\, \R^N)$ is said to be homogeneous if, for any $\lambda>0$, there exists an automorphism $D_{\lambda} : \G \to \G$ defined by
$$
D_{\lambda}(x):=(\lambda^{r_1}x_1, \lambda^{r_2} x_2, \cdots, \lambda^{r_N} x_N), \quad r_1, r_2, \cdots r_N>0.
$$
The map $D_{\lambda}$ is called a dilation on $\G$.
\end{defi}

For simplicity, we will use the notation $\lambda x$ to denote the dilation $D_{\lambda} x$. Note that if $\lambda x$ is a dilation, then $\lambda^r x$ is also a dilation. The number $Q=r_1+r_2+ \cdots+r_N$ is called the homogeneous dimension of the homogeneous Lie group $\G$ and the natural number $N$ represents the topological dimension of $\G$. The Haar measure on $G$ is denoted by $dx$ and it is nothing but the usual Lebesgue measure on $\R^N$.

\begin{defi} \label{d2}
A homogeneous Lie group $\G=(\R^N, \circ)$ is called a stratified Lie group (or a homogeneous Carnot group) if the following two conditions are fulfilled,
\begin{itemize}
\item [$(\textnormal{i})$] For some natural numbers $N_1+N_2+ \cdots +N_k=N$, the decomposition $\R^N=\R^{N_1} \times \R^{N_2} \times \cdots \times \R^{N_k}$ holds and for any $\lambda>0$, there exists a dilation of the form
$$
D_{\lambda}(x)=\left(\lambda x^{(1)}, \lambda^{2} x^{(2)}, \cdots, \lambda^k x^{(k)}\right), \quad x^{(i)} \in \R^{N_i}, i=1,2,\cdots, k,
$$
which is an automorphism of the group $\G$.
\item[$(\textnormal{ii})$] For $N_1$ the same as in the above decomposition of $\R^N$, let $X_1, X_2, \cdots, X_{N_1}$ be the left-invariant vector fields on $\G$ such that
$$
X_k(0)=\frac{\partial}{\partial x_k} \mid_0, \quad k=1, \cdots, N_1.
$$
Then the H\"ormander condition 
$$
rank \left( Lie \left\{X_1, \cdots, X_{N_1}\right\}\right)=N
$$
holds for any $x \in \R^N$. In other words, the Lie algebra corresponding to the Lie group $\G$ is spanned by the iterated commutators of $X_1, \cdots, X_{N_1}$.
\end{itemize}
\end{defi}

Here $k$ is called the step of the homogeneous Carnot group. Note that, in the case of stratified Lie groups, the homogeneous dimension becomes $Q=N_1+2N_2+ \cdots kN_k$. Furthermore, the left-invariant vector fields $X_j$ satisfy the divergence theorem and they can be written explicitly as
$$
X_i=\frac{\partial}{\partial x^{(1)}_i} + \sum_{j=2}^k \sum_{l=1}^{N_1} a_{i, l}^{(j)}(x^1, x^2, \cdots, x^{j-1}) \frac{\partial}{\partial x_l^{(j)}},
$$
where $a_{i, l}^{(j)}$ is a  homogeneous polynomal of weighted degree $j-1.$
For simplicity, we will  set $n=N_1$ in the above Definition \ref{d2}.

We would like to point out that in the literature, a stratified Lie group (or Carnot group) $\mathbb{G}$ is also defined as a connected, simply connected Lie group whose Lie algebra $\mathfrak{g}$ is stratifiable. This means that $\mathfrak{g}$ can be decomposed as a direct sum of vector spaces, $\mathfrak{g} = \bigoplus_{i=1}^k \mathfrak{g}_i$, where $[\mathfrak{g}_1, \mathfrak{g}i] = \mathfrak{g}{i+1}$ for all $i=1,2,\ldots, k-1$, and $[\mathfrak{g}_1, \mathfrak{g}_k] = {0}.$  The first layer $\mathfrak{g}_1\equiv \mathbb{R}^n$ is called the horizontal layer and it plays an important role in the theory.

It is easy to verify that any homogeneous Carnot group is a Carnot group under the classical definition (see \cite[Theorem 2.2.17]{BLU07}). The reverse implication is also true, as demonstrated in \cite[Theorem 2.2.18]{BLU07}. 

The sub-bundle of the tangent bundle $T\G,$ known as Horizontal bundle spanned by the vector fields 
$X:=\{X_1, X_2, \ldots, X_n\}$ plays a pivotal role in the theory. The fibers of $H\G$ are 
$$H\G_x:=\text{span}\{X_1(x), \ldots, X_n(x)\},\quad x \in \G. $$
It is possible to endowed each fiber of $H\G$ with a scalar product $\langle\cdot, \cdot \rangle_X$ and with a norm $|\cdot|_X$ that makes $X_1(x), \ldots, X_n(x)$ an orthonormal basis. 

Fixing the basis $X_1(x), \ldots, X_n(x)$ for the horizontal layer, we define horizontal gradient $\nabla_{\G}f$ for any function $f :\G \rightarrow \mathbb{R}$ for which the $X_jf$ exist, by 
$$\nabla_{\G}f:= \sum_{i=1}^n (X_i f)X_i,$$ whose coordinates are $(X_1f, \ldots, X_nf).$

Our operative definition of a stratified Lie group is advantageous because it is not only more practical for analytical applications, but it also provides clearer insights into the group law. Examples of stratified Lie groups include the Heisenberg group, more generally, $H$-type groups, the Engel group, and the Cartan group.

An absolutely continuous curve $\gamma: [0, 1] \to \R$ is said to be admissible, if there exist functions $c_i: [0, 1] \to \R$ for $i=1,2, \cdots, n$ such that
$$
\dot{\gamma}(t)=\sum_{i=1}^{n} c_i(t)X_i(\gamma(t)), \quad \sum_{i=1}^n c_i{t}^2 \leq 1.
$$
Observe that the functions $c_i$ may not be unique as the vector fields $X_i$ may not be linearly independent. For any $x, y \in \G$, the Carnot-Carathe\'eodory distance is defined by
$$
\rho_{cc}(x, y):=\inf \left\{ l>0: \mbox{there exists an admissible} \,\, \gamma : [0, 1] \to \G, \gamma(0)=x, \gamma(l)=y\right\}.
$$
We define that $\rho_{cc}(x, y)=0$ if no such curve exists. This $\rho_{cc}$ is not a metric in general but the H\"ormander condition for the vector fields $X_1, X_2, \cdots, X_{N_1}$ ensures that $\rho_{cc}$ is a metric. This space $(\G, \rho_{cc})$ is known as a Carnot-Carath\'eodory space $\G$.

\begin{defi}
A continuous function $|\cdot|: \G \to \R_{\geq 0}$ is said to be a homogeneous quasi-norm on a homogeneous Lie group $\G$ if it satisfies the following conditions,
\begin{itemize}
\item [$(\textnormal{i})$] (Definiteness) $|x|=0$ if and only if $x=0$.
\item [$(\textnormal{ii})$] (Symmetric) $|x^{-1}|=|x|$ for any $x \in \G$.
\item [$(\textnormal{iii})$] (1-homogeneous) $|\lambda x|=\lambda|x|$ for any $x \in \G$ and $\lambda>0$.
\end{itemize}
\end{defi}

An example of a quasi-norm on $\G$ is the norm defined by $d(x):=\rho_{cc}(x, 0)$ for $x \in \G$, where $\rho_{cc}$ is a Carnot-Carath\'eodory distance related to H\"ormander vector fields on $\G$. It is known that all homogeneous quasi-norms are equivalent on $\G$. In this paper, we will work with a left-invariant homogeneous distance
$$
d(x, y):=|y^{-1} \circ x|, \quad x, y \in \G
$$
induced by the homogeneous quasi-norm of $\G$.

Let $\Omega$ be a Haar measurable subset of $\G$. Then $\mu(D_{\lambda}(\Omega))=\lambda^{Q}\mu(\Omega)$, where $\mu(\Omega)$ is the Harr measure of $\Omega$. The quasi-ball of radius $r$ centered at $x \in \G$ with respect to the quasi-norm is defined by
$$
B_r(x):= \left\{ y \in \G : |y^{-1} \circ x| <r \right\}.
$$
Observe that $B_r(x)$ can be obtained by the left-translation by $x$ of the ball $B(0, r)$. Furthermore, $B(0, r)$ is the image under the dilation of $D_r$ of $B_1(0)$. Then we have that $\mu(B_r(x))=r^{Q}$ for any $x \in \G$. Now we recall the following result about the polar decomposition on homoegenoues Lie groups. 
\begin{lem} \label{polarc} \cite[Proposition 3.1.42]{FR}
Let $\G$ be a homogeneous Lie group with the homogeneous quasi-norm $| \cdot|$. Then there exists a  positive Borel measure $\sigma$ on the unit sphere
$$
\mathcal{S}:=\left\{ x \in \G : |x|=1\right\}
$$
such that, for any $f \in L^1(\G)$, there holds that
$$
\int_{\G} f(x) \,dx=\int_0^{\infty} \int_{\mathcal{S}} f(rx) r^{Q-1} d\sigma dr.
$$
\end{lem}
\subsection{Fractional Sobolev spaces on stratified Lie groups}
We are now in a position to introduce the notion of fractional Sobolev spaces in the framework of stratified Lie groups. Let $\Omega \subset \G$ be an open subset. The fractional Sobolev space $W^{s,p}(\Omega)$ is defined by 
$$
W^{s,p}(\Omega):=\left\{u \in L^p(\Omega) : [u]_{s, p, \Omega} < +\infty\right\}
$$
endowed with the norm
$$
\|u\|_{W^{s,p}(\Omega)}:=\|u\|_{L^p(\Omega)} + [u]_{s,p, \Omega},
$$
where $[u]_{s,p, \Omega}$ denotes the Gagliardo semi-norm defined by
\begin{align} \label{norm}
[u]_{s,p, \Omega}:=\left(\int_{\Omega} \int_{\Omega} \frac{|u(x)-u(y)|^p}{|y^{-1}x|^{Q+ps}}\,dxdy\right)^{\frac 1p}.
\end{align}
The fractional Sobolev space $W^{s,p}_0(\Omega)$ is defined by the completion of of $C_0^{\infty}(\Omega)$ under the norm $\|u\|_{W^{s,p}(\Omega)}$. It is clear to see that $W^{s,p}_0(\G)=W^{s,p}(\G)$.

Let $\Omega \subset \G$ be a bounded domain. We define the space $\X$ be the completion of $C_0^{\infty}(\Omega)$ under the norm
$$
\|u\|_{L^p(\Omega)}+[u]_{s,p, \G}.
$$
It follows from \cite[Lemma 1]{GKR} that $\X$ is a reflexive Banach space. In view of the following Poincar\'e's inequality proved in \cite{GKR},
\begin{align} \label{pi}
\|u\|_{L^p(\Omega)} \leq C [u]_{p,s,\G},
\end{align}
then $\X$ can be regarded as the completion of $C_0^{\infty}(\Omega)$ under the norm $[u]_{p,s,\G}$.
Therefore, there holds that
\begin{align} \label{defx}
X_0^{s,p}(\Omega)=\left\{u \in W^{s, p}(\G) : u=0 \,\, \mbox{in} \,\, \G \backslash \Omega \right\}
\end{align}
Next, we give definition of an extension domain in the setting of stratified Lie groups, which plays an important role in the study of fractional Sobolev spaces.
\begin{defi} \label{defextension}
Let $\Omega \subset \G$ be a bounded domain. We say that $\Omega$ is an extension domain of $W^{s,p}_0(\Omega)$ if for any $f \in W^{s,p}_0(\Omega)$ there exists $\widetilde{f} \in W^{s,p}_0(\G)$ such that $\widetilde{f}\mid_{\Omega}=f$ and
$$
\|\widetilde{f}\|_{W^{s,p}_0(\G)} \leq C_{Q,s,p,\Omega} \|f\|_{W^{s,p}_0(\Omega)},
$$
where $C_{Q,s,p,\Omega}>0$ is a constant depending on $Q,s,p$ and $\Omega$.
\end{defi}

Now, we define the Folland-Stein classes $C^{0, \alpha}(\Omega)$ for $0<\alpha \leq 1$ and an open set $\Omega \subset \G,$ by 
\begin{align}
    C^{0, \alpha}(\Omega):= \left\{ u \in L^\infty(\Omega): \underset{x, y \in \Omega, x \neq y}{\sup} \frac{|u(x)-u(y)|}{|y^{-1}x|^\alpha}< \infty \right\}
\end{align}
which equipped with the norm 
\begin{align}
    \|u\|_{C^{0, \alpha}(\Omega)}:= \|u\|_{L^\infty(\Omega)}+[u]_{C^{0, \alpha}(\Omega)},
\end{align}

where the Folland-Stein seminorm is defined as 
\begin{align}
    [u]_{C^{0, \alpha}(\Omega)}:= \underset{x, y \in \Omega, x \neq y}{\sup} \frac{|u(x)-u(y)|}{|y^{-1}x|^\alpha}.
\end{align}
Also, we define $C^{0, \alpha}(\overline{\Omega}):=   C^{0, \alpha}(\Omega) \cap C(\overline{\Omega}).$ 
Throughout the paper, we will use the notation $X \lesssim Y$ to indicate the estimate $X \leq CY$ for certain constant $C>0.$

\section{Fractional Morrey-Sobolev type embeddings on stratified Lie groups} \label{embedding}

In this section, we will  present  the proof of   Fractional Morrey-Sobolev type embeddings, that is, Theorem \ref{thm1} and Theorem \ref{thm2}. We begin with presenting the following technical result which will be useful to prove Theorem \ref{thm1} and Theorem \ref{thm2}.

\begin{lem} \label{fem} Let $\G$ be a stratified Lie group of homogeneous dimension $Q.$
Let $0<s<1$ and $1 \leq p<+ \infty$. Then, for any $u \in W^{s, p}(\G)$, there holds that
$$
\sup_{|z|>0} \int_{\G} \frac{|u(xz)-u(x)|^p}{|z|^{ps}} \,dx \leq C [u]^p_{W^{s, p}(\G)},
$$
where $C:=C(Q, p, s)>0$.
\end{lem}
\begin{proof}
Let $\chi \in C^{\infty}_0(\G)$ be a nonnegative function with the support in $B_1(0) \backslash B_{1/2}(0)$ and satisfying 
\begin{align} \label{zero}
\int_{\G} \nabla_{\G} \chi \,dx=0.
\end{align}
Without restriction, we may assume that $\|\chi\|_{L^{1}(\G)}=1$. Otherwise, we will  use the function $\chi/\|\chi\|_{L^{1}(\G)}$ instead of $\chi$. Let $z=(z', z'') \in \G \backslash \{0\}$ and $0<\eps<|z|$, where $z'$ is a horizontal component of $z$. 
Define
$$
\chi_{\eps}(\cdot)=\frac{1}{\eps^Q} \chi\left(\frac{\cdot}{\eps}\right).
$$
It is clear to see that $\|\chi_{\eps}\|_{L^1(\G)}=1$. By making changes of variables and using the fact that the Haar measure $dy$ is unimodular and translation invariant, we then obtain that
\begin{align} \label{em1}
\begin{split}
|u(xz)-u(x)|&= \left| \int_{\G} (u(xz)-u(x)) \chi_{\eps}(y) dy \right| \\&=\left|\int_{\G}u(y) \chi_{\eps}(y^{-1}(xz)) \,dy + \int_{\G} \left(u(xz)-u((xz)y^{-1})\right) \chi_{\eps}(y) \,dy \right.\\
& \quad \,\, - \left.  \int_{\G} u(y) \chi_{\eps}( y^{-1}x) \, dy - \int_{\G} \left(u(x)-u(xy^{-1} )\right) \chi_{\eps}(y)\,dy \right|  \\
& \leq \left|\int_{\G}u(y) \left(\chi_{\eps}(y^{-1}(xz))-\chi_{\eps}(y^{-1} x) \right)\,dy\right| \\&\quad\quad\quad + \int_{\G} \left|u(xz)-u((xz)y^{-1} )\right|\chi_{\eps}(y) \,dy \\
& \quad\quad\quad\quad + \int_{\G} \left|u(x)-u( xy^{-1})\right| \chi_{\eps}(y)\,dy. 
\end{split} \end{align}
We are going to estimate the first term on the right-hand side of \eqref{em1}. Thanks to \eqref{zero}, we compute that
\begin{align} \label{id1}
\int_{\G} \nabla_{G} \chi_{\eps}(x) \,dx=\frac{1}{\eps^{Q+1}} \int_{\G} \nabla_{\G} \chi_{\eps}\left(\frac{x}{\eps}\right) \,dx=0.
\end{align}
Utilizing \eqref{id1} and proceeding a similar calculation as in the proof of \cite[Lemma 3.1]{GT24} together with the Pansu differentiability \cite{Pansu} when target space is $(\mathbb{R}, +)$  , we get  the following estimate
\begin{align} \label{es1'} \nonumber
\Big| \int_{\G}u(y) &\left(\chi_{\eps}(y^{-1}(xz))-\chi_{\eps}(y^{-1} x) \right)\,dy \Big|=\Big|\int_{\G} u(y) \int_0^1 \frac{d}{ds} \chi_{\eps}(y^{-1} x(sz)) \,dsdy \Big|\\ \nonumber
&=\Big|\int_0^1 \int_{\G} u(y)\langle \nabla_{\G}\chi_{\eps}((y^{-1} x)sz), z' \rangle \,dsdy\Big|\\
& =\Big|\int_0^1 \int_{\G} \left(u(y)-u(x(sz))\right) \langle \nabla_{\G}\chi_{\eps}((y^{-1} x)sz), z' \rangle\,dsdy\Big| \\\nonumber
& \leq \frac{\|\nabla_{\G} \chi\|_{L^{\infty}(\G)} |z'|}{\eps^{Q+1}}\int_0^1 \int_{B_{\eps}(x(sz)) \backslash B_{\eps/2}(x(sz))} |u(y)-u(x(sz))| \,dsdy,
\end{align}
where $|z'|$ is the Euclidean norm of $z' \in \mathfrak{g}_1=\mathbb{R}^n.$ 
Note that there always exists (due to continuity and homogeneity) a constant $M:= \max_{|z|=1} |z'|>0$ such that  $|z'| \leq M |z|$ for any quasi-norm $|\cdot|$ on $\G.$
Thus, we deduce from \eqref{es1'} that 
\begin{align*} 
\bigg| &\int_{\G}u(y) \left(\chi_{\eps}(y^{-1}(xz))-\chi_{\eps}(y^{-1} x) \right)\,dy \bigg|  \\
&\leq \frac{\|\nabla_{\G} \chi\|_{L^{\infty}(\G)} M |z|}{\eps^{Q+1}}\int_0^1 \int_{B_{\eps}(0) \backslash B_{\eps/2}(0)} |u(x(sz)y)-u(x(sz))| \,dsdy.
\end{align*}
Applying Jensen's inequality, we then derive that
\begin{align*} 
\begin{split}
\bigg | & \int_{\G}u(y)\left(\chi_{\eps}(y^{-1}(xz))-\chi_{\eps}(y^{-1} x) \right)\,dy \bigg |^p \\ &\leq \frac{\|\nabla_{\G} \chi\|^p_{L^{\infty}(\G)} M^p |z|^p}{\eps^{p}} \left(\frac{1}{\eps^Q}\int_0^1 \int_{B_{\eps}(0) \backslash B_{\eps/2}(0)} |u(x(sz)y)-u(x(sz))| \,dsdy\right)^p \\
&\leq \frac{\|\nabla_{\G} \chi\|^p_{L^{\infty}(\G)} M^p |z|^p}{\eps^{Q+p}}\int_0^1 \int_{B_{\eps}(0) \backslash B_{\eps/2}(0)} |u(x(sz)y)-u(x(sz))|^p \,dsdy.
\end{split}
\end{align*}
Similarly, invoking the definition of $\chi_{\eps}$ and Jensen's inequality, we are able to show that
\begin{align}
\left(\int_{\G} \left|u(xz)-u((xz)y^{-1} )\right|\chi_{\eps}(y) \,dy\right)^p &\leq\frac{\|\chi\|^p_{L^{\infty}(\G)}}{\eps^{Q}}\int_{B_{\eps}(0) \backslash B_{\eps/2}(0)}  \left|u(xz)-u((xz)y^{-1} )\right|^p \,dy,\\
\left(\int_{\G}  \left|u(x)-u(xy^{-1})\right| \chi_{\eps}(y)\,dy \right)^p &\leq\frac{\|\chi\|^p_{L^{\infty}(\G)}}{\eps^{Q}}\int_{B_{\eps}(0) \backslash B_{\eps/2}(0)}  \left|u(x)-u(xy^{-1})\right| ^p \,dy.
\end{align}
Therefore, taking into account \eqref{em1} and the change of variables along with the translation invariance of the Haar measure, we obtain that
\begin{align*}
\int_{\G}|u(xz)-u(x)|^p \,dx & \lesssim \frac{\|\nabla_{\G} \chi\|^p_{L^{\infty}(\G)} |z|^p}{\eps^{Q+p}}\int_{B_{\eps}(0) \backslash B_{\eps/2}(0)}\int_{\G} |u(xy)-u(x)|^p \,dxdy \\
& \quad +\frac{\|\chi\|^p_{L^{\infty}(\G)}}{\eps^{Q}}\int_{B_{\eps}(0) \backslash B_{\eps/2}(0)} \int_{\G} |u(xy)-u(x)|^p \,dxdy.
\end{align*}
Since $0<\eps<|z|$, we then get that 
\begin{align*}
\int_{\G}&|u(xz)-u(x)|^p \,dx \\&\lesssim \frac{\left(\|\nabla_{\G} \chi\|^p_{L^{\infty}(\G)} + \|\chi\|^p_{L^{\infty}(\G)}\right)|z|^p}{\eps^{Q+p}}\int_{B_{\eps}(0) \backslash B_{\eps/2}(0)}\int_{\G} |u(xy)-u(x)|^p \,dxdy.
\end{align*}
Then it  follows that
\begin{align*}
\int_{\G}&\frac{|u(xz)-u(x)|^p}{|z|^{ps}} \,dx \\&\lesssim \frac{\left(\|\nabla_{\G} \chi\|^p_{L^{\infty}(\G)} + \|\chi\|^p_{L^{\infty}(\G)}\right)|z|^{p(1-s)}}{\eps^{Q+p}}\int_{B_{\eps}(0) \backslash B_{\eps/2}(0)}\int_{\G} |u(xy)-u(x)|^p \,dxdy.
\end{align*}
Let $\eps>|z|/2$, then
$$
\frac{|z|^{p(1-s)}}{\eps^{Q+p}} \leq \frac{2^{p(1-s)}} {\eps^{Q+ps}} \leq \frac{2^{p(1-s)}} {|y|^{Q+ps}}, \quad \forall \,\, y \in B_{\eps}(0) \backslash B_{\eps/2}(0).
$$
This implies that
\begin{align*}
\int_{\G}&\frac{|u(xz)-u(x)|^p}{|z|^{ps}} \,dx 
\\&\lesssim 2^{p(1-s)}\left(\|\nabla_{\G} \chi\|^p_{L^{\infty}(\G)} + \|\chi\|^p_{L^{\infty}(\G)}\right) \int_{B_{\eps}(0) \backslash B_{\eps/2}(0)}\int_{\G} \frac{|u(xy)-u(x)|^p}{|y|^{Q+ps}} \,dxdy\\
& \leq C \int_{\G}\int_{\G} \frac{|u(x)-u(y)|^p}{|y^{-1}x|^{Q+ps}} \,dxdy. 
\end{align*}
Hence the desirable conclusion follows and the proof is completed.
\end{proof}

Relying on Lemma \ref{fem}, we are now ready to present the proof of Theorem \ref{thm1}.

\begin{proof}[Proof of Theorem \ref{thm1}] Let $p \leq r<+\infty$. Since $ps=Q$, then there exists $0<s'<s<1$ such that $s'p<Q$ and $r<p_{s'}^*$. It then follows from Theorem \ref{thm0} that
\begin{align} \label{em11}
\|u\|_{L^r(\Omega)} \leq C \|u\|_{X^{s', p}_0(\Omega)}.
\end{align}
Let $u \in X^{s, p}_0(\Omega)$, by \eqref{defx}, then $u \in W^{s, p}(\G)$ and $u=0$ in $\G \backslash \Omega$. Observe that
\begin{align}\label{em21}
\begin{split}
[u]^p_{s', p, \G}&= \int_{\G} \int_{\G} \frac{|u(x)-u(y)|^p}{|y^{-1} x|^{Q+ps'}} \,dxdy=  \int_{\G} \int_{\G} \frac{|u(xz)-u(x)|^p}{|z|^{Q+ps'}} \,dxdz\\
&=\int_{\left\{z \in \G : |z|>1\right\}} \int_{\G} \frac{|u(xz)-u(x)|^p}{|z|^{Q+ps'}} \,dxdz\\&\quad\quad\quad\quad\quad+\int_{\left\{z \in \G : |z|\leq 1\right\}} \int_{\G} \frac{|u(xz)-u(x)|^p}{|z|^{Q+ps'}} \,dxdz.
\end{split}
\end{align}
By Lemma \ref{polarc} and Poincar\'e's inequality \eqref{pi}, one finds that
\begin{align*}
\int_{\left\{z \in \G : |z|>1\right\}} \int_{\G}& \frac{|u(xz)-u(x)|^p}{|z|^{Q+ps'}} \,dxdz 
\\ &\leq 2^{p-1} \int_{\left\{z \in \G : |z|>1\right\}}  \frac{1}{{|z|^{Q+ps'}}} \left(\int_{\G} |u(xz)|^p + |u(x)|^p  \,dx \right) dz \\
& \lesssim \int_{\Omega} |u|^p \,dx \lesssim  [u]_{s,p, \G}.
\end{align*}
On the other hand, by Lemmas \ref{polarc} and \ref{fem}, we have that
\begin{align*}
\int_{\left\{z \in \G : |z|\leq 1\right\}} \int_{\G}& \frac{|u(xz)-u(x)|^p}{|z|^{Q+ps'}} \,dxdz
\\&=\int_{\left\{z \in \G : |z|\leq 1\right\}} \left( \int_{\G} \frac{|u(xz)-u(x)|^p}{|z|^{ps}} \,dx \right) \frac{dz}{|z|^{Q-(s-s')p}}\\
& \lesssim [u]^p_{W^{s,p}(\G)}  \int_{\left\{z \in \G : |z|\leq 1\right\}}\frac{dz}{|z|^{Q-(s-s')p}} \lesssim  [u]^p_{W^{s,p}(\G)} .
\end{align*}
Going back to \eqref{em21}, we then get that $ [u]_{s',p, \G} \leq C [u]_{s,p, \G}$. It then follows from \eqref{em11} that
$$
\|u\|_{L^r(\Omega)} \leq C \|u\|_{X^{s, p}_0(\Omega)}.
$$
This clearly means that $X_0^{s,p}(\Omega)$ is continuously embedded in $L^r(\Omega)$ for any $p \leq r < + \infty$. 

Next we will   demonstrate that the embedding is compact for any $p \leq r <+\infty$. Let $\{u_n\} \subset X^{s, p}_0(\Omega)$ be a bounded sequence. Let $0<s<s'<1$ be such that $ps'<Q$ and $r<p_{s'}^*$. Then $\{u_n\}$ is bounded in $X^{s',p}(\Omega)$. By Theorem \ref{thm0}, then $\{u_n\}$ is compact in $L^r(\Omega)$. Then we have the desirable conclusion and the proof is completed.
\end{proof}

In the following, we are going to investigate the Morrey-Sobolev type embedding for the case $sp>Q$. We first have the following result.

\begin{lem} \label{fem1}
Let $\G$ be a stratified Lie group of homogeneous dimension $Q$ and let $\Omega \subset \G$ be an open subset. Let $0<s<1 \leq p<+\infty$ and $sp>Q$. Then there exists a constant $C=C(Q, p, s)>0$ such that, for any $u \in W^{s, p}(\G)$,
$$
\|u\|_{C^{0, \alpha}(\G)} \leq C \left(\|u\|_{L^p(\G)} + [u]_{s,p, \G}\right), \quad \alpha=\frac{ps-Q}{p}.
$$
\end{lem}
\begin{proof}
To establish this lemma, we will   take advantage of the elements presented in the proof of \cite[Theorem 8.2]{DPV}. Define
$$
u_{x_0, r}:=\frac{1}{|B_r(x_0)|}\int_{B_{r}(x_0)} u(x) \,dx,
$$
$$
[u]_{p, ps}:=\left(\sup_{x_0 \in \G, r>0} \frac{1}{r^{ps}}\int_{B_r(x_0)} |u-u_{x_0, r}|^p \,dx \right)^{\frac 1p}.
$$
Let us first claim that there exists a constant $C=C(Q, s, p)>0$ such that, for any $x_0 \in \G$ and $0<r'<r<+\infty$, there holds that
\begin{align} \label{ineq}
|u_{x_0, r}-u_{x_0, r'}| \leq C[u]_{p,ps} |B_r(x_0)|^{\frac{ps-Q}{Qp}}.
\end{align}
To prove this claim, let us take $0<\rho_1<\rho_2<r$. Then, by H\"older's inequality, we see that
\begin{align} \label{em31}
\begin{split}
|u_{x_0, \rho_1}-u_{x_0, \rho_2}| &\leq \frac{1}{|B_{\rho_1}(x_0)|} \int_{B_{\rho_1}(x_0)} |u(x)-u_{x_0, \rho_2}| \,dx \\
& \leq \frac{1}{|B_{\rho_1}(x_0)|^{\frac 1p}}\left(\int_{B_{\rho_1}(x_0)} |u(x)-u_{x_0, \rho_2}|^p \,dx \right)^{\frac 1p} \\
& \leq  \frac{1}{|B_{\rho_1}(x_0)|^{\frac 1p}}\left(\int_{B_{\rho_2}(x_0)} |u(x)-u_{x_0, \rho_2}|^p \,dx \right)^{\frac 1p} \\
& \lesssim \left(\frac{\rho_2}{\rho_1}\right)^s |B_{\rho_1}(x_0)|^{\frac{ps-Q}{Qp}}  [u]_{p, ps},
\end{split}
\end{align}
where we have used the fact that $|B_{\rho_1}(x_0)| \approx \rho_1^Q.$

Define $r_i=2^{-i} r$, $i \in \mathbb{N} \cup \{0\}$. It then follows from \eqref{em31} that
$$
|u_{x_0, r_{i+1}}-u_{x_0, r_i}| \lesssim |B_{r_{i+1}}(x_0)|^{\frac{ps-Q}{Qp}} [u]_{p, ps}.
$$
As a result, there holds for $k \in \N$ that
\begin{align} \label{ineq1}
|u_{x_0, r_{k}}-u_{x_0, r}| \lesssim |B_{r}(x_0)|^{\frac{ps-Q}{Qp}} [u]_{p, ps}.
\end{align}
We now choose $k \in \N$ such that $r_k < r'\leq r_{k-1}<r$. Therefore, by \eqref{em31}, we get that
$$
|u_{x_0, r'}-u_{x_0, r_k}| \lesssim \left(\frac{r'}{r_k}\right)^s |B_{r_k}(x_0)|^{\frac{ps-Q}{Qp}} [u]_{p, ps} \lesssim |B_{r}(x_0)|^{\frac{ps-Q}{Qp}} [u]_{p, ps},
$$
which along with \eqref{ineq1} plainly leads to the desired conclusion \eqref{ineq}.

At this point, we now take into account \eqref{ineq} to prove Lemma \ref{fem1}. Let $u \in W^{s,p}(\G)$. Using Jensen's inequality, we know that
$$
\int_{B_r(x_0)} |u(x)-u_{x_0, r}|^p \,dx \leq \frac{1}{|B_r(x_0)|} \int_{B_r(x_0)} \int_{B_r(x_0)} |u(x)-u(y)|^p\,dxdy.
$$
Observe that if $x, y \in B_r(x_0)$, then $|y^{-1}x| \leq 2r$. Therefore, we are able to conclude that
\begin{align} \label{em33}
\int_{B_r(x_0)}& |u(x)-u_{x_0, r}|^p\,dx \nonumber \\ &\leq  \frac{1}{|B_r(x_0)|} \int_{B_r(x_0)} \int_{B_r(x_0)} \frac{|u(x)-u(y)|^p}{|y^{-1}x|^{Q+ps}}|y^{-1}x|^{Q+ps}\,dxdy \lesssim r^{ps} [u]_{s,p, \G}.
\end{align}
It then immediately follows that $[u]_{p, ps} \lesssim [u]_{p, s, \G}$. Consequently, by \eqref{ineq}, we have that
\begin{align} \label{emlimit}
|u_{x, r}-u_{x, r'}| \lesssim [u]_{p,s, \G} |B_r(x)|^{\frac{ps-Q}{Qp}}.
\end{align}
This indicates that $\lim_{r \to 0} u_{x, r}$ exists uniformly for $x \in \G$. Passing the limit as $r'$ goes to zero and applying \eqref{emlimit}, we then get that
\begin{align} \label{ineq3}
|u_{x, 2r}-u(x)| \lesssim [u]_{p,s, \G} |B_r(x)|^{\frac{ps-Q}{Qp}}.
\end{align}
Let $x, x' \in \G$ be such that $|x^{-1} x'|=r$. Then there holds that
\begin{align} \label{em34}
|u(x)-u(x')| \leq |u(x)-u_{x, 2r}| + |u(x')-u_{x',2r}|+|u_{x, 2r}-u_{x', 2r}|.
\end{align}
Let $x'' \in B_{2r}(x) \cap B_{2r}(x')$. Observe that
$$
|u_{x, 2r}-u_{x', 2r}| \leq |u(x'')-u_{x, 2r}| +|u(x'')-u_{x', 2r}|.
$$
It then follows that
\begin{align} \label{estr}
|u_{x, 2r}-u_{x', 2r}| &\leq \frac{1}{|B_{2r}(x)|}\int_{B_{2r}(x)} |u(x'')-u_{x, 2r}| \,dx''\nonumber \\&\quad\quad +\frac{1}{|B_{2r}(x')|}\int_{B_{2r}(x)} |u(x'')-u_{x', 2r}| \,dx''.
\end{align}
As an application of H\"older's inequality and \eqref{em33} implies that
$$
 \frac{1}{|B_{2r}(x)|}\int_{B_{2r}(x)} |u(x'')-u_{x, 2r}| \,dx'' \lesssim r^{-\frac{Q}{p}} \left(\int_{B_{2r}(x)} |u(x'')-u_{x, 2r}|^p \,dx''\right)^{\frac 1p} \lesssim r^{\frac{ps-Q}{p}}[u]_{s, p, \G}.
$$
Similarly, there holds that
$$
 \frac{1}{|B_{2r}(x)|}\int_{B_{2r}(x)} |u(x'')-u_{x, 2r}| \,dx'' \lesssim r^{\frac{ps-Q}{p}}[u]_{s, p, \G}.
$$
This together with \eqref{estr} implies that
$$
|u_{x, 2r}-u_{x', 2r}| \lesssim r^{\frac{ps-Q}{p}}[u]_{s, p, \G}.
$$
Consequently, exploiting \eqref{ineq3} and \eqref{em34}, we arrive at
$$
|u(x)-u(x')| \lesssim r^{\frac{ps-Q}{p}}[u]_{s, p, \G}= |x^{-1} x'|^{\frac{ps-Q}{p}}[u]_{s, p, \G}.
$$
This clearly infers that
\begin{align} \label{h1}
[u]_{C^{0, \alpha}(\G)}:=\sup_{x \neq x'} \frac{|u(x)-u(x')|}{ |x^{-1} x'|^{\frac{ps-Q}{p}}} \lesssim [u]_{s, p, \G}.
\end{align}
In addition, by \eqref{ineq3} and H\"older's inequality, there holds that, for some $r_0>0$,
\begin{align} \label{h2}
\begin{split}
|u(x)| & \lesssim \left(\frac{1}{|B_{2r_0}(x)|} \int_{B_{2r_0}(x)}|u(y)|\,dy + |B_{r_0}(x)|^{\frac{ps-Q}{Qp}} [u]_{s,p,\G}\right) \\
& \lesssim \left(\frac{1}{|B_{2r_0}(x)|^{\frac 1p}} \left(\int_{B_{2r_0}(x)}|u(y)|^p\,dy\right)^{\frac 1p} + |B_{r_0}(x)|^{\frac{ps-Q}{Qp}} [u]_{s,p,\G}\right) \\
& \lesssim \|u\|_{L^p(\G)} + [u]_{s,p,\G}.
\end{split}
\end{align}
Hence the desirable conclusion follows by combining \eqref{h1} and \eqref{h2}. Indeed, we have 
\begin{align*}
     \|u\|_{C^{0, \alpha}(\G)}&= \|u\|_{L^\infty(\G)}+[u]_{C^{0, \alpha}(\G)} 
     \leq C (\|u\|_{L^p(\G)} + [u]_{s,p,\G}),
\end{align*}
 completing the proof.
\end{proof}

To show that the embedding for the case $sp>Q$ is compact, we need to introduce the following result, which can be regarded as a general version of the well-known Arzel\'a-Ascoli theorem.

\begin{lem} \label{AA} \cite[Proposition 3.2]{B}
Suppose that $X$ is a compact set and $U \subset C(X)$ is a bounded subset such that, for any $x \in X$ and $\eps>0$, there exists a neighborhood of $N$ of $x$ such that $|u(x)-u(y)| \leq \eps$ for any $y \in N$ and $u \in U$. Then any sequence in $U$ has a uniformly convergent subsequence.
\end{lem}

Building upon the results established previously, we are now in a position to present the proof of Theorem \ref{thm2}.

\begin{proof}[Proof of Theorem \ref{thm2}] 
Let $u \in X^{s, p}_0(\Omega)$, by \eqref{defx}, then $u \in W^{s, p}(\G)$ and $u=0$ in $\G \backslash \Omega$. Using Lemma \ref{fem1}, we then have that $u \in C^{0, \alpha}(\Omega)$ and
$$
\|u\|_{C^{0, \alpha}(\Omega)} \lesssim \left(\|u\|_{L^p(\Omega)} + [u]_{s, p, \G}\right).
$$
Then, by Poincar\'e's inequality \eqref{pi}, we obtain that
$$
\|u\|_{C^{0, \alpha}(\Omega)} \lesssim \|u\|_{X^{s,p}_0(\Omega)}.
$$
This obviously infers that $X_0^{s,p}(\Omega)$ is continuously embedded into $C^{0, \alpha}(\Omega)$ proving \eqref{emthm2}.

Next we are going to show that the embedding $X^{s,p}_0(\Omega)$ is compactly embedded into $C^{0, \beta}(\overline{\Omega})$ is compact for any $0<\beta<\alpha$ when $\Omega$ is further assumed to be bounded. We first prove that $X^{s,p}_0(\Omega)\hookrightarrow C(\bar{\Omega})$ is compact. We will use the Arzela-Ascoli theorem and Lemma \ref{AA}. Let $K$ be a bounded subset of $X^{s,p}_0(\Omega)$. Therefore, there exists a constant $C>0$ such that for all $u\in K$, we have
\begin{equation*}
    \|u\|_{L^{\infty}(\Omega)}\leq C \|u\|_{X_{0}^{s,p}(\Omega)}.
\end{equation*}
Moreover, for all $x,y\in\Omega$, we have
\begin{equation*}
    |u(x)-u(y)|\leq C \|u\|_{X_{0}^{s,p}(\Omega)}|y^{-1}x|^{\alpha}.
\end{equation*}
Therefore, we get $K$ is bounded in $L^{\infty}(\Omega)$ as well as $K$ is equicontinuous. Thus using Lemma \ref{AA}, we conclude the embedding $X^{s,p}_0(\Omega)\hookrightarrow C(\bar{\Omega})$ is compact.\\
Next, we show that the embedding $X^{s,p}_0(\Omega)\hookrightarrow C^{0, \beta}(\Omega)$ is compact for any $0<\beta<\alpha$. Let $\{u_n\} \subset X^{s,p}_0(\Omega)$ be a bounded sequence. Then $\{u_n\}$ is bounded in $C^{0, \alpha}(\Omega)$. In addition, we see that
$$
\sup_{x \in \Omega} |u_n(x)| \leq C \|u_n\|_{X^{s,p}_0(\Omega)}, \quad \sup_{x, y \in \Omega} \frac{|u_n(x)-u_n(y)|}{|y^{-1}x|^{\alpha}} \leq C \|u_n\|_{X^{s,p}_0(\Omega)}.
$$
Invoking Lemma \ref{AA}, we then have that $\{u_n\}$ has a uniformly convergent subsequence in $C(\Omega)$. Without restriction, we may assume that $u_n \to u$ in $C(\Omega)$ as $n \to \infty$. Observe that
$$
\frac{|(u_n-u)(x)-(u_n-u)(y)|}{|y^{-1}x|^{\beta}}=\lim_{m \to \infty}\frac{|(u_n-u_m)(x)-(u_n-u_m)(y)|}{|y^{-1}x|^{\beta}} \leq C |y^{-1}x|^{\alpha-\beta},
$$
$$
\frac{|(u_n-u)(x)-(u_n-u)(y)|}{|y^{-1}x|^{\beta}} \leq \frac{2}{|y^{-1}x|^{\beta}} \|u_n-u\|_{C(\Omega)}.
$$
As a consequence, we conclude that $\|u_n-u\|_{C^{0, \beta}(\Omega)}=o_n(1)$. Therefore, we conclude that $X^{s,p}_0(\Omega)\hookrightarrow C^{0, \beta}(\bar{\Omega})$ is compact for any $0<\beta<\alpha$. This completes the proof.
\end{proof}

\section{Applications: Subelliptic nonlocal problems involving oscillating nonlinearities} \label{application}

In this section, we will present applications of the embedding results to subelliptic nonlocal problems involving oscillating nonlinearities. The main purpose of this section is to give the proof Theorems \ref{mainthmintrnon} and \ref{color1.1}. In fact, to establish the existence and asymptotic behaviors of solutions to \eqref{pro1intro}, we are going to adapt the variational principle relying on Ricceri \cite{Ric00}.

\begin{defi}
We say that a function $u \in X_{0}^{s, p}(\Omega)\backslash\{0\}$ is a (weak) solution to \eqref{pro1intro} if 
\begin{align*}
\int_{\G}\int_{\G} \frac{|u(x)-u(y)|^{p-2}(u(x)-u(y))(\varphi(x)-\varphi(y))}{|y^{-1}x|^{Q+sp}} \, dx dy = \lambda \int_{\Omega} \beta(x) g(u) \varphi \, dx
\end{align*} 
holds for any $\varphi \in X^{s, p}_0(\Omega)$.
\end{defi}

The underlying energy function $\mathfrak{I}_\lambda :X^{s, p}_0(\Omega) \rightarrow \mathbb{R}$ is defined by 
\begin{equation}
\mathfrak{I}_\lambda(u) := \frac{1}{p}\int_{\G}\int_{\G} \frac{|u(x)-u(y)|^{p}}{|y^{-1}x|^{Q+sp}}\,dx dy - \lambda \int_{\Omega} \beta(x) G(u)\,dx,
\end{equation} 
where 
$$
G(\tau):= \int_{0}^\tau g(\xi)\,\, d\xi, \quad \forall \,\, \tau \in \R.
$$
Since $\beta \in L^\infty(\Omega)$ and $g$ is continuous, we then derive that $\mathfrak{I}_\lambda \in C^1(X^{s, p}_0(\Omega), \R)$ and its derivative is given by 
\begin{align*}
\langle \mathfrak{I}'_\lambda(u), \varphi \rangle =\int_{\G}\int_{\G}  \frac{|u(x)-u(y)|^{p-2}(u(x)-u(y))(\varphi(x)-\varphi(y))}{|y^{-1}x|^{Q+sp}} dx dy-\lambda \int_{\Omega} \beta(x) g(u) \varphi dx 
\end{align*}
for any $\varphi \in X^{s, p}_0(\Omega)$. It immediately follows that solutions to \eqref{pro1intro} correspond to critical points of the energy function $\mathfrak{I}_\lambda$. Therefore, to prove the existence of solutions, it is sufficient to find critical points of $\mathfrak{I}_\lambda$ via variational methods. To this end, we will make use of the following abstract result due to Bonanno and Bisci \cite{BB, BB1}, which is a refinement of the variational principle developed by Ricceri \cite{Ric00}. 

\begin{thm}\label{ricthm} 
Let $X$ be a reflexive real Banach space and let $\Phi,\Psi:X\to\R$ be two G\^{a}teaux differentiable functionals
such that $\Phi$ is strongly continuous, sequentially weakly lower semicontinuous and coercive, $\Psi$ is sequentially weakly upper semicontinuous. For every $r>\inf_X \Phi,$ put
$$
\varphi(r):=\inf_{u \in \Phi^{-1}((-\infty, r))}\frac{\displaystyle\sup_{v \in ({\Phi^{-1}((-\infty, r))})}\Psi(v)-\Psi(u)}{r-\Phi(u)},
$$
$$
\gamma:=\liminf_{r \to +\infty}\varphi(r),\quad \delta:=\liminf_{r \to (\inf_X \Phi)^+}\varphi(r).
$$
Then the following conclusions hold:
\begin{itemize}
\item[$(\textnormal{i})$] If $\gamma<+\infty$, then, for every $0<\lambda<{1}/{\gamma}$, the following alternative holds:\\
{\rm $(a)$} either $\Phi-\lambda \Psi$ possesses a global minimum,\\
{\rm $(b)$} or there is a sequence of critical points $\{u_n\}$ (local minima) of $\Phi-\lambda \Psi$ such that 
$$
\lim_{n\to \infty} \Phi(u_n)=+\infty.
$$
\item[$(\textnormal{ii})$] If $\delta<+\infty$, then, for every $0<\lambda<{1}/{\delta}$, the following alternative holds:\\
{\rm $(a)$} either there is a global minimizer of $\Phi$ which is a local minimum of $\Phi-\lambda \Psi$,\\
{\rm $(b)$} or there is a sequence of pairwise distinct critical points $\{u_n\}$  (local minima) of $\Phi-\lambda\Psi$
such that
$$
\displaystyle\lim_{n\to \infty} \Phi(u_n)=\inf_{u\in X}\Phi(u).
$$
\end{itemize}
\end{thm}

In the following, we will apply Theorem \ref{ricthm} to show the existence of solutions to \eqref{pro1intro}. Let $\Omega$ be an open subset of $\G$ and $B_{\tau}(x_0)$ be the ball center at $x_0 \in \G$ and of radius $\tau>0.$ As $\Omega$ is open subset of $\G$, then it is possible to choose $x_0 \in \Omega$ and $\tau>0$ such that $B_{\tau}(x_0) \subset \Omega$. We set $x_0 \in \G$ as the Chebyshev center of $\Omega$ and $r:=\sup_{x \in \Omega} \text{dist}(x, \partial \Omega).$ For the fixed $x_0$ and $r>0$, we define a function $T: \G \to \R$ by
\begin{align} \label{testfun}
T(x):= \begin{cases} 
0 & \text{if}\quad x \in \mathbb{G} \setminus B_{r}(x_0),\\
1 & \text{if} \quad x \in B_{\frac{r}{2}}(x_0),\\
\frac{2(r-|x_0^{-1}x|)}{r} &\text{if}\quad x \in B_r(x_0) \setminus B_{\frac{r}{2}}(x_0),
\end{cases}
\end{align}
where $|\cdot|$ denotes the quasi-norm on $\G.$ Since $T$ is zero outside of the compact ball $B_r(x_0) \subset \Omega$, we then see that $T \in X^{s,p}_0(\Omega).$

\begin{lem}\label{lmest}
Let $T \in X_0^{s,p}(\Omega)$ be the function defined by \eqref{testfun}. Then the following estimate holds true,
\begin{equation}\label{est}
\int_{\G}\int_{\G} \frac{|T(x)-T(y)|^p}{|y^{-1}x|^{Q+sp}} dx dy \leq C_{p, Q, s} \omega_Q^2 r^{Q-sp}, 
\end{equation}
where $\omega_Q$ is the volume  of the unit sphere with respect to the quasi-norm $|\cdot |$ on $\G$ and 
\begin{equation*} 
C_{p, Q, s}:=\left(\frac{2^{p(2-s)}}{(1-s)p}+\frac{2^{2+ps}}{ps}\right)\left(1-\frac{1}{2^Q}\right)+\frac{2^{1+ps-Q}}{ps}.
\end{equation*}
\end{lem}
\begin{proof}
We are going to prove the estimate \eqref{est} by dividing the integral domain $\mathbb{G} \times \mathbb{G}$ into the subdomains. Note that 
\begin{align*}
\mathbb{G} \times \mathbb{G}=&(B_r \times B_r)\cup(B_r^c \times B_r)\cup(B_r \times B_r^c)\cup(B_r^c \times B_r^c)\\
=&(B_{\frac{r}{2}} \times B_{\frac{r}{2}})\cup(B_{\frac{r}{2}} \times (B_r\setminus B_{\frac{r}{2}}))\cup((B_r\setminus B_{\frac{r}{2}}) \times B_{\frac{r}{2}})\\
&\cup((B_r\setminus B_{\frac{r}{2}}) \times (B_r\setminus B_{\frac{r}{2}}))\cup(B_r^c \times B_{\frac{r}{2}})\cup(B_r^c \times (B_r\setminus B_{\frac{r}{2}}))\\
&\cup(B_{\frac{r}{2}} \times B_r^c)\cup((B_r\setminus B_{\frac{r}{2}}) \times B_r^c)\cup(B_r^c \times B_r^c),
\end{align*}
where $B_r:=B_r(x_0)$ and $B_r^c:=\G\setminus B_r(x_0)$. Accordingly, using the definition of $T$ and the symmetry of $x$ and $y$, we get that
\begin{equation}\label{est1}
\int_{\G}\int_{\G}\frac{|T(x)-T(y)|^p}{|y^{-1}x|^{Q+sp}}\,dxdy=I_1+I_2+I_3+I_4,
\end{equation}
where 
$$
I_1:=\int_{B_{r}\setminus B_{{\frac{r}{2}}}}\int_{B_{r}\setminus B_{{\frac{r}{2}}}}\frac{|T(x)-T(y)|^p}{|y^{-1}x|^{Q+sp}}\,dxdy,
$$
$$
I_2:={2} \int_{B_{r}\setminus B_{{\frac{r}{2}}}}\int_{B^c_{r}} \frac{|T(x)-T(y)|^p}{|y^{-1}x|^{Q+sp}}\,dxdy,
$$
$$
I_3:={2}\int_{B_{{\frac{r}{2}}}}\int_{B_{r}\setminus B_{{\frac{r}{2}}}} \frac{|T(x)-T(y)|^p}{|y^{-1}x|^{Q+sp}}\,dxdy,
$$
$$
I_4:=2\int_{B_{{\frac{r}{2}}}}\int_{B^c_{r}}\frac{|T(x)-T(y)|^p}{|y^{-1}x|^{Q+sp}}\,dxdy.
$$
To achieve the desirable conclusion, we need to estimate the terms $I_1$, $I_2$, $I_3$ and $I_4$. We begin with treating the term $I_1$. Observe that, for any $(x, y)\in (B_{r}\setminus B_{{\frac{r}{2}}})\times(B_{r}\setminus B_{{\frac{r}{2}}})$,
$$
\left||x_0^{-1}x|-|x_0^{-1}y|\right|\leq |y^{-1}x|,
$$
$$
\int_{B_{r}\setminus B_{ {\frac{r}{2}}}} \frac{1}{|y^{-1}x|^{Q-(1-s)p}} \, dx \leq \int_{B_{r+|y^{-1}x_0|}(0)\setminus B_{|y^{-1}x_0|-\frac r 2}(0)} \frac{1}{|z|^{Q-(1-s)p}} \,dz \leq \int_{B_{2r}(0)} \frac{1}{|z|^{Q-(1-s)p}} \,dz.
$$
Therefore, we are able to derive that
\begin{align}\label{i1}
\begin{split}
I_1 &=\int_{B_{r}\setminus B_{{\frac{r}{2}}}}\int_{B_{r}\setminus B_{{\frac{r}{2}}}}\frac{|T(x)-T(y)|^p}{|y^{-1}x|^{Q+sp}}\,dxdy \\
&=\frac{2^p}{r^p} \int_{B_{r}\setminus B_{{\frac{r}{2}}}}\int_{B_{r}\setminus B_{{\frac{r}{2}}}} \frac{\left||x_0^{-1}x|-|x_0^{-1} y|\right|^p}{|y^{-1}x|^{Q+sp}} \, dx dy \\
&\leq \frac{2^p}{r^p} \int_{B_{r}\setminus B_{{\frac{r}{2}}}}\int_{B_{r}\setminus B_{ {\frac{r}{2}}}} \frac{|y^{-1}x|^p}{|y^{-1}x|^{Q+sp}} \, dx dy\\
& \leq  \frac{2^p}{r^p} \int_{B_{r}\setminus B_{{\frac{r}{2}}}}\int_{B_{2r}(0)} \frac{1}{|z|^{Q-(1-s)p}} \,dzdy\\
&=\frac{2^{p(2-s)}\omega_Q^2}{(1-s)p}\left(1-\frac{1}{2^Q}\right) r^{Q-ps}.
\end{split}
\end{align}
Let us now deal with the term $I_2$. We first note that, for any $y\in B_{r}\setminus B_{{\frac{r}{2}}}$, 
\begin{equation}
\int_{B_{r}^c}\frac{1}{|y^{-1}x|^{Q+sp}}\,dx \leq \int_{B_{r-|y^{-1}x_0|}^c(0)}\frac{1}{|z|^{Q+sp}}\,dx.
\end{equation}
Consequently, we conclude that
\begin{align}\label{i2}
\begin{split}
I_2 &= \frac{2^{p+1}}{r^p}{\int_{B_{r}\setminus B_{\frac r2}}\int_{B_{r}^c}\frac{\left|r -|x_0^{-1}y|\right|^p}{|y^{-1}x|^{Q+ps}}}\,dx dy \\
&\leq \frac{2^{p+1}\omega_Q}{r^p}{\int_{B_{r}\setminus B_{\frac r2}}\int_{r-|y^{-1}x_0|}^{+\infty}\frac{\left|r -|x_0^{-1}y|\right|^p}{\xi^{1+ps}}} \,d\xi dy \\
&=\frac{2^{p+1}\omega_Q}{ps r^p} \int_{B_{r}\setminus B_{\frac r2}} \left|r -|y^{-1}x_0|\right|^{(1-s)p} \,dy \\
& \leq \frac{2^{1+ps}\omega_Q^2}{ps}\left(1-\frac{1}{2^Q}\right) r^{Q-ps},
\end{split}
\end{align}
where we used the fact that $r-|x_0^{-1}y| \leq \frac r 2$ for any $y \in B_{r}\setminus B_{\frac r 2}$. We now turn to handle the term $I_3$. We first note that, for any $x\in B_{r}\setminus B_{{\frac{r}{2}}}$,
\begin{equation*} 
\int_{B_{\frac r2}}\frac{1}{|y^{-1}x|^{Q+ps}} \,dy  \leq \int_{B_{|x_0^{-1}x|+\frac r 2}(0) \setminus B_{|x_0^{-1}x|-\frac r 2}(0) }\frac{1}{|z|^{Q+ps}} \,dz \leq \int_{B^c_{|x_0^{-1}x|-\frac r 2}(0) }\frac{1}{|z|^{Q+ps}} \,dz
\end{equation*}
Similarly, arguing as the proof of \eqref{i2}, we have that
\begin{align}\label{i3}
\begin{split}
I_3 & = \frac{2^{p+1}}{r^p}{\int_{B_{\frac r2}}\int_{B_{r}\setminus B_{\frac r2}}\frac{\left||x_0^{-1}x|-\frac r2\right|^p}{|x^{-1}y|^{Q+ps}}}\,dxdy \\
& \leq \frac{2^{p+1}\omega_Q}{r^p}\int_{B_{r}\setminus B_{\frac r2}}\int_{B_{|x_0^{-1}x|-\frac r 2}^c(0)}\frac{\left||x_0^{-1}x|-\frac r2\right|^p }{|z|^{Q+ps}} \,dz\, dx \\
& = \frac{2^{p+1}\omega_Q}{psr^p} \int_{B_{r}\setminus B_{ {r/2}}}\left||x_0^{-1}x|-\frac r 2\right|^{(1-s)p}\,dx \\
& \leq \frac{2^{1+ps}\omega_Q^2}{ps}\left(1-\frac{1}{2^Q}\right) r^{Q-ps}.
\end{split}
\end{align}
We next estimate the term $I_4$. Observe that, for any $y \in B_{\frac r 2}$,
\begin{equation*} 
\int_{B_{r}^c}\frac{1}{|y^{-1}x|^{Q+ps}}\,dx \leq \int_{B_{r -|y^{-1}x_0|}^c(0)}\frac{1}{|z|^{Q+ps}}\,dz.
\end{equation*}
Thereby, we are able to obtain that
\begin{align}\label{i4}
I_4 & =  2{\int_{B_{\frac r 2}}\int_{B_{r}^c}\frac{1}{|y^{-1}x|^{Q+ps}}}\,dxdy \\
& \leq 2 \int_{B_{\frac r 2}}\int_{B_{r -|y^{-1}x_0|}^c(0)}\frac{1}{|z|^{Q+ps}}\,dz dy\\
&= \frac{2\omega_Q}{ps}\int_{B_{\frac r 2}} \frac{1}{\left(r-|x_0^{-1}y|\right)^{ps}} \,dy \\
& \leq \frac{2^{1+ps-Q}\omega_Q^2}{ps}r^{Q-ps}.
\end{align}
Finally, invoking \eqref{i1}, \eqref{i2}, \eqref{i3} and \eqref{i4}, we then derive the estimate \eqref{est}.
This completes the proof.    
\end{proof}

Now, relying on Lemma \ref{lmest}, we are going to present the proof of Theorem \ref{mainthmintrnon}.

\begin{proof}[Proof of Theorem \ref{mainthmintrnon}] 
As mentioned earlier, we will  adapt Theorem \ref{ricthm} to prove this theorem. Here the space $X^{s, p}_0(\Omega)$ is equipped with the norm given by 
$$ 
\|u\|_{X^{s, p}_0(\Omega)}:= \left(\int_{\G}\int_{\G} \frac{|u(x)-u(y)|^{p}}{|y^{-1}x|^{Q+sp}} dx dy\right)^{\frac{1}{p}},\quad u \in \X.
$$
To make use of Theorem \ref{ricthm}, we define respectively the functionals $\Phi: \X \rightarrow \R$ and $\Psi : \X \rightarrow \R$ by
$$
\Phi(u):=  \frac{1}{p}\int_{\G}\int_{\G} \frac{|u(x)-u(y)|^{p}}{|y^{-1}x|^{Q+sp}}\, dxdy,
$$
$$
\Psi(u):=\int_{\Omega} \beta(x) G(u)\,dx.
$$
And we set $J_{\lambda}:= \Phi -\lambda \Psi$. 

It is easy to see that the functional $\Phi$ is G\^{a}teaux differentiable and its G\^{a}teaux derivative is given by 
$$
\langle \Phi'(u), \varphi \rangle:= \int_{\G}\int_{\G} \frac{|u(x)-u(y)|^{p-2}(u(x)-u(y))(\varphi(x)-\varphi(y))}{|y^{-1}x|^{Q+sp}} \,dx dy, \quad \forall \,\,\varphi \in  \X.
$$ 
Obviously, $\Phi$ is strongly continuous, sequentially weakly lower semicontinuous, coercive and 
$\inf_{u \in \X} \Phi(u)=0$.
In addition, the function $\Psi$ is G\^{a}teaux differentiable and its G\^{a}teaux derivative is given by
$$
\langle \Psi'(u), \varphi \rangle:=\int_{\Omega} \beta(x) g(u) \varphi \,dx, \quad \forall\,\, \varphi \in \X.
$$ 
Next, we are going to show that $\Psi$ is sequentially weakly upper semicontinuous by utilizing Theorem \ref{emthm2}. Indeed, by the arguments developed in \cite{BB, Ric00}, we only need to prove that $\Psi$ is sequentially weakly upper semicontinuous up to subsequences. Let $\{u_n\} \subset \X$ be such that $u_n \wto u_0$ an $n \to \infty$ in $\X$. 
It follows from Theorem \ref{thm1} that there exists $C>0$ such that $\|u_n\|_{C(\overline{\Omega})} \leq C$ and $u_n \to u$ as $n \to \infty$ in $C(\overline{\Omega})$.
Since $G$ is continuous, then $G(u_n(x)) \to G(u(x))$ as $n \to \infty$ for any $x \in \G$.
Observe that  
$$
|\beta(x) G(u_n(x))| \leq \|\beta\|_{L^\infty(\Omega)} \max_{|\tau| \leq C} |G(\tau)|, \quad \forall \,\, x \in \Omega.
$$
Employing the Lebesgue dominated convergence theorem, we then have that $\lim_{n \to \infty} \Psi(u_n)=\Psi(u)$. This in turn leads to the desirable conclusion.

Let $\{a_n\} \subset \R_{>0}$ be such that 
$$
\lim_{n \rightarrow \infty} a_n =M, \quad  \lim_{n \rightarrow \infty} \frac{\max_{|\tau| \leq a_n} G(\tau)}{a_n^p} =A_M, \quad M \in \{0^+, +\infty\},
$$
where
$$
A_M:= \liminf_{\tau \rightarrow M} \frac{\max_{|\xi| \leq \tau} G(\xi)}{\tau^p}.
$$
Define $b_n:= \frac{a_n^p}{p L^p}>0$, where
$$
L:=\sup\left\{\frac{\|u\|_{L^\infty(\Omega)}}{\|u\|_{X_0^{s,p}(\Omega)}}: u \in X_0^{s, p}(\Omega) \backslash \{0\} \right\}>0.
$$ 
It is clear to see that $\lim_{n \to \infty} b_n=M$. Applying the fact that $\X$ is continuously embedded into $C(\overline{\Omega})$ by Theorem \ref{thm1}, we then derive that $L<\infty$. Then, for any $v \in \Phi^{-1}(-\infty, b_n)$, there holds that
\begin{align} \label{anest}
\|v\|_{L^\infty(\Omega)} \leq a_n, \quad \forall \,\, n \in \Z_{>0}.
\end{align}
It then yields that, for any $v \in \Phi^{-1}(-\infty, b_n)$,
\begin{align} \label{anest1}
\Psi(v)=\int_{\Omega} \beta(x) G(v)\,dx \leq \|\beta\|_{L^{\infty}(\Omega)} |\Omega| \max_{|\tau| \leq a_n} G(\tau).
\end{align}
Since $\Phi(0)=\Psi(0)=0$, then $0 \in \Phi^{-1}((-\infty, b_n))$. This along with \eqref{anest} and \eqref{anest1} leads to
\begin{align}
\varphi(b_n)= \inf_{u\in\Phi^{-1}((-\infty, b_n))}\frac{\displaystyle\sup_{v\in\Phi^{-1}((-\infty,b_n))}\Psi(v)-\Psi(u)}{b_n-\Phi(u)} 
&\leq \frac{\displaystyle\sup_{v\in\Phi^{-1}((-\infty,b_n))}\Psi(v)}{b_n} \\
&\leq p \|\beta \|_{L^\infty(\Omega)} |\Omega| L^p \frac{\max_{|\tau| \leq a_n} G(\tau)}{a_n^p}.
\end{align}
Therefore, if $0<\lambda < \lambda_{2, M}$, then
$$
\alpha_M \leq \liminf_{n \rightarrow \infty} \varphi(b_n) \leq p \|\beta\|_{L^\infty(\Omega)} |\Omega| L^p A_M <\frac{1}{\lambda}<+\infty,
$$
where 
$$
\lambda_{2, M}:=\frac{1}{p \|\beta\|_{L^\infty(\Omega)} |\Omega| L^p A_M},
$$
$$
\alpha_M:= \begin{cases}
\gamma:= \liminf_{r \rightarrow +\infty} \varphi(r)\quad &\text{if} \quad M=+\infty \\ \delta:= \liminf_{r \rightarrow 0^+} \varphi(r)\quad & \text{if} \quad M=0^+. 
\end{cases}
$$
It then indicates that, if $0<\lambda < \lambda_{2, M}$, then $0<\lambda<1/\alpha_M$.

Now, we are going to show that, if $M=+\infty$, then the function $J_\lambda$ is unbounded from below. Let $\{\xi_n\} \subset \R$ be such that $\lim_{n\rightarrow \infty} \xi_n =+\infty$ and 
\begin{equation} \label{eqq4.11}
\lim_{n \rightarrow \infty} \frac{G(\xi_n)}{\xi_n^p}=B_\infty,
\end{equation}
where
$$
B_{\infty}:=\limsup_{\tau \to +\infty} \frac{G(\tau)}{\tau^p}.
$$
Define $W_n:= \xi_n T \in \X$, where the function $T: \G \to \R$ is defined by \eqref{testfun}. In view of Lemma \ref{lmest}, we readily have that 
\begin{align} \label{w1}
\Phi(W_n) = \frac{1}{p} \int_{\mathbb{G} \times \mathbb{G}} \frac{|W_n(x)-W_n(y)|^p}{|y^{-1}x|^{Q+sp}}\, dx dy \leq C_{p, Q, s} \omega_Q^2 \frac{r^{Q-sp}}{p} \xi_n^p.
\end{align} 
Moreover, since we assumed that $\inf_{\tau \geq 0} G(\tau)=0$ and $\beta_0:= \text{essinf}_{x \in \Omega} \beta (x)>0$, then there holds that
\begin{equation} \label{w2}
\Psi(W_n)= \int_\Omega \beta(x) G(W_n) \,dx \geq \beta_0 \int_{B_{\frac{r}{2}}(x_0)} G(W_n) \,dx \geq  \beta_0 \omega_Q  \frac{r^Q}{2^Q} G(\xi_n)>0.
\end{equation} 
Therefore, by \eqref{w1} and \eqref{w2}, we conclude that 
\begin{equation} \label{e4.12}
J_\lambda(W_n)=\Psi(W_n)-\lambda \Psi(W_n) \leq C_{p, Q, s} \omega_Q^2 \frac{r^{Q-sp}}{p} \xi_n^p- \lambda \beta_0  \omega_Q \frac{r^Q}{2^Q} G(\xi_n).
\end{equation} 
Let $\lambda>\lambda_{1, \infty}$, then
$$
0<\lambda':=\frac{2^Q C_{p, Q, s} \omega_Q }{\lambda p r^{sp} \beta_0 B_\infty}<1,
$$
where
$$
\lambda_{1, \infty}:=\frac{2^Q C_{p, Q, s} \omega_Q }{p r^{sp} \beta_0 B_\infty}.
$$
If $B_\infty<+\infty$, we then take $\eps>0$ satisfy $\lambda'<\eps<1$.
It then follows from \eqref{eqq4.11} that there exists $N_\epsilon \in \Z_{>0}$ such that 
$$
G(\xi_n) > \epsilon B_\infty \xi^p_n, \quad \forall\,\, n>N_{\eps}.
$$
Therefore, by \eqref{e4.12}, we conclude that 
\begin{align*}
J_\lambda(W_n) 
\leq \left(C_{p, Q, s} \frac{\omega_Q}{p r^{ps}}-\lambda \epsilon B_\infty \frac{\beta_0}{2^Q}  \right) r^Q \omega_Q \xi_n^p, \quad \forall\,\, n>N_{\eps}.
\end{align*}
Since $\eps>\lambda'$, then
$\lim_{n\rightarrow \infty} J_\lambda(W_n)= - \infty$.
If $B_\infty=+\infty$, we then take $\widetilde{\eps}>0$ satisfy $\widetilde{\eps}>\widetilde{\lambda}$, where
$$
\widetilde{\lambda}:=\frac{2^Q C_{p, Q, s} \omega_Q}{\lambda p r^{ps} \beta_0}.
$$
Again, using \eqref{eqq4.11}, we know that that there exists $N_{\widetilde{\eps}} \in \Z_{>0}$ such that 
$$
G(\xi_n)>\widetilde{\eps} \xi_n^p, \quad \forall\,\, n>N_{\eps'}.
$$ 
Therefore, applying \eqref{e4.12}, we get that
\begin{align*}
J_\lambda(W_n) 
\leq \left(C_{p, Q, s}  \frac{\omega_Q}{r^{sp} p} -\lambda \widetilde{\eps} \frac{\beta_0}{2^Q} \right) r^Q \omega_Q \xi_n^p, \quad \forall\,\, n>N_{\widetilde{\eps}}.
\end{align*}
Since $\widetilde{\eps}>\widetilde{\lambda}$, then $\lim_{n\rightarrow \infty} J_\lambda(W_n)= - \infty$.
At this point, taking into account the assertion $(\textnormal{i})$ of Theorem \ref{ricthm}, we have that the functional $J_\lambda$ admits an unbounded sequence of critical points $\{u_{\lambda, n}\} \subset \X$ such that 
$$
\lim_{n \rightarrow \infty} \|u_{\lambda, n}\|_{X^{s,p}_0(\Omega)}=+\infty.
$$
Next, we are going to consider the case $M=0^+$. In this case, by similar arguments as above, we can show that $u_0=0$ is not a local minimum of $J_\lambda$. Thereby, taking advantage of the assertion $(\textnormal{ii})$ of Theorem \ref{ricthm}, we conclude that the functional $J_\lambda$ admits a sequence of the pairwise distinct critical point (local minima) $\{u_{\lambda, n}\} \subset \X$ such that 
$$
\lim_{n \rightarrow \infty} \|u_{\lambda, n}\|_{X^{s,p}_0(\Omega)}=0.
$$
As a consequence of Theorem \ref{thm2}, we get that $\lim_{n \rightarrow \infty} \|u_{\lambda, n}\|_{L^\infty(\Omega)}=0$. This completes the proof.  
\end{proof}

Finally, we are going to present the proof of Theorem \ref{color1.1}.

\begin{proof}[Proof of Theorem \ref{color1.1}] Since $g$ is a nonnegative function with $g(0)=0,$ we then adapt Theorem \ref{mainthmintrnon} to the nonlocal subelliptic problem \begin{equation} \label{corsec}
\begin{cases} 
(-\Delta_{\mathbb{G}, p})^s u= \lambda \beta(x) g^+(u), & \text{in} \quad \Omega \\
u=0\quad & \text{in}\quad \mathbb{G}\backslash \Omega,
\end{cases}
\end{equation}
where $g^+:\R \rightarrow \R$ is defined by 
$$
g^+(\tau):= \begin{cases}
g(\tau) \quad & \mbox{if}\,\,\tau>0 \\
0 \quad & \mbox{if}\,\, \tau \leq 0,
\end{cases}
$$
which is continuous for any $\tau \in \R$. Using the nonnegativity of $g$, we then see that 
$$ 
\max_{|\xi| \leq \tau} \int_{0}^\xi g^+(s)\,ds = \max_{|\xi| \leq \tau} \int_0^\xi g(w)\,ds =G(\tau), \quad \forall\,\, \tau \geq 0.
$$ 
At this point, using the assumptions,  we then get $A_M=0$ and $B_M=+\infty.$ Therefore, the conclusion of  Theorem \ref{mainthmintrnon} holds with $\lambda_1^M =0$ and $\lambda_2^M=+\infty.$ Hence, to conclude the proof, we have to show that the solutions are nonnegative. For this, we set $a^{\pm}:=\max \{0, \pm a \}$ for any $a \in \R.$ Let $u \in \X$ be a solution to \eqref{corsec}. Then there holds that $u^\pm \in \X$ and 
\begin{align} \label{ee}
\int_{\Omega}\beta(x) g^+(u) u^-(x)=0.
\end{align}
Observe that
\begin{align} \label{ee1}
|a^--b^-|^p \leq |a-b|^{p-2}(a-b)( b^- -a^- ), \quad \forall \,\, a, b \in \R.
\end{align}
Taking into account \eqref{ee} and \eqref{ee1} 
we then derive that 
\begin{align*}
0= \langle \mathfrak{I}_\lambda'(u), -u^- \rangle 
&= \int_{\G}\int_{\G} \frac{|u(x)-u(y)|^{p-2}(u(x)-u(y))(u^-(y)-u^-(x))}{|y^{-1}x|^{Q+sp}}\, dx dy \\
& \geq \int_{\G}\int_{\G}  \frac{|u^-(x)-u^-(y)|^p}{|y^{-1}x|^{Q+sp}}\,dx dy,
\end{align*}
which implies that $u^-=0$ on $\G.$ Therefore, we have that $u \geq 0$. This completes the proof. 
\end{proof}

\section*{Declarations}
\subsection*{Data availability}

Data sharing not applicable to this article as no datasets were generated or analysed during the current study.

\subsection*{Conflict of interest}

The authors confirm that there is no conflict of interest.

\subsection*{Acknowledgments} 
T. Gou was supported by the National Natural Science Foundation of China (No. 12101483) and the Postdoctoral Science Foundation of China (No. 2021M702620). T. Gou thanks Prof. Louis Jeanjean for useful discussions on applications of the embeddings. S. Ghosh acknowledges the research facilities available at the Department of Mathematics, NIT Calicut under the DST-FIST support, Govt. of India [Project no. SR/FST/MS-1/2019/40 Dated. 07.01 2020.].  V. Kumar is supported by the FWO Odysseus 1 grant G.0H94.18N: Analysis and Partial Differential Equations, the Methusalem programme of the Ghent University Special Research Fund (BOF) (Grant number 01M01021) and by FWO Senior Research Grant G011522N. V.D. R\u adulescu was supported  by   grant ``Nonlinear Differential Systems
in Applied Sciences" of the Romanian Ministry of Research, Innovation and Digitization, within
PNRR-III-C9-2022-I8/22.

\end{document}